\documentclass[12pt]{amsart}

\usepackage{amsmath,amssymb,amsfonts,graphics,amsrefs}
\usepackage{fullpage}
\usepackage{hyperref}
\hypersetup{
    colorlinks=true,
    linkcolor=blue,
    filecolor=magenta,
    urlcolor=cyan,
}

\newtheorem{thm}{Theorem}[section]
\newtheorem{cor}[thm]{Corollary}

\newtheorem{prop}[thm]{Proposition}

\theoremstyle{definition}
\newtheorem{defn}[thm]{Definition}
\theoremstyle{remark}
\newtheorem{rem}[thm]{Remark}
\numberwithin{equation}{section}

\newcommand{\bb}[1]{\mathbb{#1}}
\newcommand{\cl}[1]{\mathcal{#1}}

\begin{document}

\title[Amenable traces and joint numerical range]{Amenable traces and the joint numerical radius}

\author[Vern I. Paulsen]{Vern I. Paulsen}
\address{Institute for Quantum Computing and Department of Pure Mathematics,
University of Waterloo, Waterloo, ON, Canada N2L 3G1}
\email{vpaulsen@uwaterloo.ca}

\author[Mizanur Rahaman]{Mizanur Rahaman}
\address{Wallenberg Centre for Quantum Technology, Chalmers University of Technology;
Department of Mathematical Sciences, Chalmers University of Technology and University of
Gothenburg}
\email{mizanurr@chalmers.se}

\author[Ebrahim Samei]{Ebrahim Samei}
\address{Department of Mathematics and Statistics, University of Saskatchewan, Saskatoon, SK, S7N 5E6, Canada}
\email{ebrahim.samei@usask.ca}

\begin{abstract}
We provide necessary and sufficient characterizations of the existence of an amenable trace on a C$^*$-algebra in terms of the joint free numerical radius of tuples of unitaries, isometries, and partial isometries in the algebra. We apply these results to obtain new obstructions to various lifting properties.
\end{abstract}

\maketitle


\section{Introduction and Preliminaries}

The main goal of this paper is to give a number of characterizations of when a unital C*-algebra enjoys an amenable trace in terms of the {\it joint free numerical radius} which was introduced in \cite{FKP}.

\begin{defn}\label{D:free joint numerical radius}
Let $\cl A\subseteq B(H)$ be a unital C$^*$-algebras, and let $a_1,\dots,a_n\in \cl A$. The {\it free joint numerical radius} $w_{cb}(a_1,\dots,a_n)$ (also referred as the cb-numerical radius) is defined by
\begin{align*}
 w_{cb}(a_1,\dots,a_n)=\sup \{w(U_1\otimes a_1+\dots+U_n\otimes a_n) \},
\end{align*}
where the supremum is taken over every $k$, every choice of $n$ contractions
$U_1,\dots , U_n \in M_k$. 
\end{defn}

By considering compressions, it is not hard to see that the quantity remains unchanged if we allow $U_1,\ldots, U_n \in B(K)$ for an arbitrary Hilbert space and the tensor product is the spatial tensor or if we restrict to $n$-tuples of unitaries instead of contractions.

This quantity was introduced in \cite{FKP} where it was used to give an extension of a classical result of Ando's that characterized operators $T$ with $w(T) \le 1/2$ and then used to give a new characterization of C*-algebras with the weak expectation property(WEP). Those ideas were further extended in \cite{FKPT1, FKPT2} to give some new tensor product and order theoretic characterizations of WEP. Most recently the joint free numerical radius was used in \cite{DPR} to give new proofs of some exactness and lifting properties for group C*-algebras. We further extend some of those ideas in section~4.

Among the results that we prove is that a unital C*-algebra has an amenable trace if and only if for any generating $n$-tuple $a_1,...,a_n$ in the algebra, consisting of isometries or co-isometries, one has $w_{cb}(a_1,\ldots,a_n) =n$. We also give necessary and sufficient conditions for a C*-algebra generated by partial isometries to have an amenable trace.

Amenable traces are important in C$^*$-algebra theory because they capture finite-dimensional approximation of tracial states and link C$^*$-algebraic structure with von Neumann algebraic approximation phenomena such as $\mathcal{R}^\omega$-embeddability (\cite{Brown}). They are also relevant to quantum information theory through the study of quantum correlations, Tsirelson-type problems, and nonlocal games (\cite{Junge-et-al}), themes that played a decisive role in the negative resolution of the Connes Embedding Problem via the breakthrough result $\mathrm{MIP}^*=\mathrm{RE}$ (\cite{MIP}).

We begin by recalling some key facts about amenable traces.

 
\begin{defn}
    Let $\cl A\subseteq B(H)$ be a C$^*$-algebra. A tracial state $\tau$ on $\cl A$ is called \textit{amenable} provided that there is a state $\rho$ on $B(H)$ such that $\rho|_{\cl A}=\tau$ and $\rho(aT)=\rho(Ta)$ for all $T\in B(H)$ and $a\in \cl A$.
\end{defn}

The state $\rho$ is often referred to as an \textit{$\cl A$-hypertrace}, or more simply as a hypertrace, when the algebra $\cl A$ is understood. An application of Arveson's Extension Theorem shows that amenability of $\tau$ is independent of the faithful representation of $\cl A$. There are several other characterizations of amenable traces coming from the work of Connes and Kirchberg. Brown's monograph \cite{Brown} is an excellent source for this material.  We recall a few of the main results from there. 

Let $\cl R$ denote the hyperfinite $\mathrm{II}_1$ factor, if $\omega$ is a free ultrafilter over the positive integers, then $\cl R^{\omega}$ is the corresponding tracial ultrapower.  Combining \cite[Theorem~3.1.6]{Brown} and \cite[Theorem~3.1.7]{Brown}, we have the following characterizations of amenable traces.

\begin{thm}\label{T:Charcrization amenable trace-Kirchberg}
    Let $\cl A$ be a separable C$^*$-algebra, and let $\tau$ be a tracial state on $\cl A$. Then the following statements are equivalent:
    \begin{enumerate}
        \item $\tau$ is amenable;
        \item There is a $*$-homomorphism $\pi: \cl A\rightarrow \cl R^{\omega}$ with a completely positive, contractive lift $\cl A\rightarrow \ell^{\infty}(\cl R)$ such that $tr_{\cl R^{\omega}}\circ\pi=\tau$, where $tr_{\cl R^{\omega}}$ is the canonical tracial state on $\cl R^{\omega}$;
        \item there is a sequence of completely positive, contractive maps $\phi_k: \cl A\rightarrow M_{d_k}$ such that
        \[\|\phi_k(ab)-\phi_k(a)\phi_k(b)\|_{2}\to 0, \ \text{and} \ tr_{d_k}(\phi_k(a))\to \tau(a),\]
        for all $a,b\in \cl A$.
    \item The linear functional $\phi:\cl A\otimes \cl A^{op}\rightarrow \mathbb{C}$ defined by $\phi(a\otimes b^{op})=\tau(ab)$ extends to a bounded linear functional on $\cl A \otimes_{min} \cl A^{op}$. 
    \item The linear functional $\phi$ extends to a state on $\cl A \otimes_{min} \cl A^{op}$.
    \end{enumerate}
\end{thm}

We end this subsection highlighting our main result of this article that relates the joint free numerical radius and existence of amenable trace on a C$^*$-algebra.

\begin{thm}$($Theorem \ref{T:amne trace vs FJNR-all equivalent conditions}$)$ 
Let $\cl A$ be a unital C*-algebra. Then the following are equivalent:
\begin{enumerate}
\item $\cl A$ has an amenable trace,
\item for every $n$-tuple of unitaries $(u_1,\ldots,u_n)$ in $\cl A$ and for every $n \in \bb N$ we have that $w_{cb}(u_1,\ldots,u_n) =n$,
\item for every $n$-tuple $(x_1,\ldots,x_n)$ in $\cl A$ such that each $x_i$ is either an isometry or coisometry, we have that $w_{cb}(x_1,\ldots,x_n) =n$,
\item for every $n$-tuple of unitaries $(u_1,\ldots,u_n)$ in $\cl A$ and for every $n \in \bb N$ we have that
$w(\sum_{i=1}^n u_i \otimes (u_i^*)^{op}) =n$ as an element of $\cl A \otimes_{min} \cl A^{op}$,
\item for every $n$-tuple $(x_1,\ldots,x_n)$ in $\cl A$ such that each $x_i$ is either an isometry or coisometry, we have that
$w\left(\sum_{i=1}^n x_i \otimes (x_i^*)^{op}\right)=n$ as an element of $\cl A \otimes_{min} \cl A^{op}$.
\end{enumerate}
Moreover, if $\cl A$ is generated by $(x_1,\ldots,x_n)$ where each $x_i$ is either an isometry or coisometry, then $\cl A$ has an amenable trace if and only if $w_{cb}(x_1,\ldots,x_n) =n$.
\end{thm}






\subsection{Operator systems and LLP}\label{S:Operator systems and LLP}
We now recall a few facts about operator systems and present a somewhat simpler proof of the key properties of the joint free numerical radius.
An operator system $\cl S$ is said to have the {\it lifting property(LP)} if whenever $\phi: \cl S \to \cl A/J$ is a unital completely positive map(UCP), there exists a {\it UCP lift}, i.e., a UCP map $\psi: \cl S \to \cl A$ such that $\pi \circ \psi = \phi$ where $\pi: \cl A \to \cl A/J$ is the quotient $*$-homomorphism.  An operator system is said to have the {\it local lifting property(LLP)} if whenever $\phi: \cl S \to \cl A/J$ is UCP and $\cl T \subseteq \cl S$ is a finite dimensional operator subsystem, then the restriction of $\phi$ to $\cl T$ has a UCP lift.


Given an operator system $\cl S$ the dual space $\cl S^d$ of bounded linear functionals is a $*$-vector space with $f^*(x):= \overline{f(x^*)}$ and a natural family of positive cones in $M_n(\cl S^d)$, by defining $(f_{i,j}) \in M_n(\cl S^d)^+$ if and only if the map $\phi: \cl S \to M_n$ defined by
\[
 \phi(x) = (f_{i,j}(x)) \in M_n,
\]
is completely positive. When $\cl S$ is finite-dimensional, Choi-Effros \cite[Corollary~4.5]{ChoiEffros} proved that there always exists a {\it strictly positive} linear functional, i.e., a functional $f: \cl S \to \bb C$ such that 
\[
 p \in \cl S^+, \,\, p \ne 0 \implies f(p) >0,
\]
and that any such functional satisfies the axioms to be an Archimedean order unit for $\cl S^d$, so that $(\cl S^d, f)$ satisfies the axioms to be an abstract operator system. Thus, the duals of finite-dimensional operator systems are again operator systems, albeit non-canonically, since it depends on the choice of an order unit.


Let
$\bb F_n$ the free group on $n$ generators, let $C^*(\bb F_n)$ denote its full group C*-algebra and we let $u_1,\dots,u_n \in C^*(\bb F_n)$ denote the images of the canonical generators of $\bb F_n$. Then we let set
\[\cl S_n = span \{ 1, u_1,\ldots,u_n, u_1^*,\ldots,u_n^* \} \subseteq C^*(\bb F_n).\]
We let $\delta_0, \delta_1,..., \delta_n, \delta_1^*, \delta_n^* \in \cl S_n^d$ denote the corresponding dual functionals.  Note that this notation is consistent, since $\delta_i^* = (\delta_i)^*$. Moreover this functionals are linear independent and span $\cl S_n^d$. 

We let $\cl U_n \subseteq M_2 \oplus \dots M_2$($n$ copies) be the operator system spanned by the identity $I = I_2 \oplus \cdots \oplus I_2$ and the matrix units $E_{1,2,j}, E_{2,1,j}$ where the $j$ indicates that they are in the $j$-copy of $M_2$, then it is proven in \cite{FP} that $\cl U_n$ is completely order isomorphic to $\cl S_n^d$. We present a simplified proof of this fact below, along with a result from \cite[Lemma 2.5]{DPR} that avoids the use of use operator system quotients and also indicates the role of the joint free numerical radius.

It is easily checked that given an operator on a Hilbert space,  
\[I + e^{i \theta} T + e^{-i \theta}T^* \ge 0, \forall \theta \iff w(T) \le 1/2.\]

\begin{prop}\label{wcb-prop} Let $A_1,...,A_n \in M_p$. Then the following are equivalent:
\begin{enumerate}
\item $w_{cb}(A_1,...,A_n) \le 1/2$,
\item $P:=I_p \otimes 1 + \sum_{i=1}^n A_i \otimes u_i + \sum_{i=1}^n A_i^* \otimes u_i^* \in M_p(C^*(\bb F_n))^+$,
\item the map $\Phi:\cl S_n^d \to M_p$ sending $\delta_0 \to I_p, \delta_i \to A_i, \delta_i^* \to A_i^*$ is CP,
\item the map $\Psi: \cl U_n \to M_p$ sending $I \to I_p, E_{1,2,i} \to A_i, E_{2,1,i} \to A_i^*$ is UCP.
\end{enumerate}
Consequently, the map $\Gamma: \cl S_n^d \to \cl U_n$ given by $\delta_0 \to I, \delta_i \to E_{1,2,i}, \delta_i^* \to E_{2,1,i}$ is a complete order isomorphism and with respect to this isomorphism, we have that $\delta_0$ is the order unit for $\cl S_n^d$.
\end{prop}
\begin{proof} The equivalence of (1) and (2) is clear.  
For the rest of the proof we will use the fact that  for $A_0 >0$,
\[ A_0 \otimes 1 + \sum_{i=1}^n A_i \otimes u_i + \sum_{i=1}^n A_i^* \otimes u_i^* \ge 0 \iff
I_p \otimes 1 + \sum_{i=1}^n A_0^{-1/2}A_iA_0^{-1/2} \otimes u_i + \sum_{i=1}^n A_0^{-1/2}A_i^*A_0^{-1/2} \otimes u_i^* \ge 0,\] and similarly for sums of the $\delta_i$'s, to reduce checking positivity to the case where $A_0 = I$.

Next we claim that a map $\Gamma: \cl S_n \to M_d$ sending $u_i \to B_i$ is UCP if and only if $\|B_i \| \le 1, \forall i$. To see this first note that since a UCP map is contractive, we must have $\|B_i \| \le 1$. To see the converse, let $U_i = \begin{pmatrix} B_i & \sqrt{I - B_iB_i^*} \\ \sqrt{I- B_i^*B_i} & -B_i^* \end{pmatrix}$ be the Halmos unitary dilation of $B_i$. Since the map $u_i \to U_i$ extends to a *-homomorphism of $C^*(\bb F_n)$, we see that $\Gamma$ is a compression of a *-homomorphism and so UCP.

 Hence, 
\[ \Gamma := I_p \otimes \delta_0 + \sum_{i=1}^n B_i \otimes \delta_i + \sum_{i=1}^n B_i^* \otimes \delta_i^*  \in M_p(\cl S_n^d)^+ \iff \|B_i \| \le 1, \forall i.\] 
Thus,  $\Phi: \cl S_n^d \to M_p$ is CP, if and only if, 
\[ id_d \otimes \Phi(\Gamma) = I_d \otimes I_p + \sum_{i=1}^n B_i \otimes A_i + \sum_{i=1}^n B_i^* \otimes A_i^* \ge 0, \forall \Gamma \iff w_{cb}(A_1,...,A_n) \le 1/2,\]
and we have the equivalence of (1) and (3).

Finally, since $\begin{pmatrix} I & B \\ B^* & I \end{pmatrix} \ge 0 \iff \|B\| \le 1$ we have that
\[ I _d\otimes I + \sum_{i=1}^n B_i \otimes E_{1,2,i} + \sum_{i=1}^n B_i^* \otimes E_{2,1,i} \in M_d(\cl U_n)^+ \iff \|B_i \| \le 1,\]
which shows the complete order isomorphism between $\cl U_n$ and $\cl S_n^d$, from which the equivalence of (3) and (4) follows.
\end{proof}


\section{C*-algebras generated by (co)isometries}

In this section we derive necessary and sufficient conditions for C*-algebras generated by unitaries, isometries and co-isometries to have an amenable trace in terms of the joint numerical radius of the generating tuple. First, we will show that obtaining the maximal value for the joint numerical radius of a tuple of contractions implies that the generating C*-algebra must have an amenable trace. 

\begin{prop}\label{P:amenable trace gen by unitaries} 
Let $a_1,\ldots,a_n$ be contractions on a Hilbert space $H$, and let $\cl A$ be the C*-algebra that they generate.
If $w_{cb}(a_1,\ldots,a_n) = n$, then $\cl A$ has an amenable trace.
\end{prop}
\begin{proof} Let $\epsilon >0$. We can assume that we have $n$ unitaries, $w_1,\ldots, w_n$  in a matrix algebra $M_k$ and a unit vector $h \in \bb C^k \otimes H$ such that $\langle (w_i \otimes a_i) h \vert h \rangle \ge 0$ for all $i$ and
\[n - \epsilon < \sum_{i=1}^n \langle (w_i \otimes a_i) h \vert h \rangle \leq n-1+\min\{\langle (w_i \otimes a_i) h \vert h \rangle: i=1,\ldots,n\}.\]
This implies that for $i=1,\ldots,n$,
$$1 - \epsilon < \langle (w_i \otimes a_i) h \vert h \rangle,$$ 
and so,
\[\| (w_i \otimes a_i) h -h \|^2 < 2 \epsilon \implies \|(I_k \otimes a_i)h - (w_i^* \otimes I_H) h \|^2 < 2 \epsilon.\]
Similarly, 
\[ \| (I_k \otimes a_i)h - (w_i \otimes I_H) h \|^2 < 2 \epsilon.\]
Let $s_h: B(H) \to \bb C$ be the state defined by 
$$s_h(x) = \langle (I_k \otimes x) h \vert h \rangle \ \ \ (x\in B(H)).$$
Then, for every $x\in B(H)$ and  $i=1,\ldots,n$,
\begin{align*}
& \ \ \ \ \ |s_h(xa_i) - s_h(a_ix)| \\ &= | \langle (I_k \otimes xa_i)h - (I_k \otimes a_ix) h \vert h \rangle| \\ 
& \le | \langle (I_k \otimes xa_i)h - (I_k \otimes x)(w_i^* \otimes I_H)h \vert h \rangle| +  | \langle (w_i^* \otimes I_H)(I_k\otimes x)  h - (I_k \otimes a_ix) h \vert h \rangle \\
& = |\langle (I_k \otimes x)[I_k \otimes a_i -w_i^* \otimes I_H]h \vert h \rangle|  + | \langle (I_k \otimes x) h \vert [w_i \otimes I_H - I_k \otimes a_i^*] h \rangle | \\
& \le 2 \sqrt{2 \epsilon} \|x\|. 
\end{align*}
Taking a weak*-limit point, as $\epsilon \to 0$, of the set of such states $\{s_h\}$ yields an $\cl A$-hypertrace on $B(H)$ and hence, an amenable trace on $\cl A$.

\end{proof}

We can show the converse of the preceding result under the assumption that the operators are either isometries or co-isometries.  

\begin{thm}\label{T:amen trace imply max FJNR}
If $\cl A \subseteq B(H)$ has an amenable trace, then for every $n \in \bb N$ and every $n$-tuple of operators $(x_1,...,x_n)$ in $\cl A$ such that each $x_i$ is either an isometry or coisometry, we have that $w_{cb}(x_1,...,x_n) =n$.
\end{thm}
\begin{proof}
Let $\tau$ be an amenable trace on $\cl A,$ then it extends to a hypertrace  $\psi$ on $B(H)$.
Assume that $\cl A$ has an $n$-tuple as above with
$w_{cb}(x_1,...,x_n)=w <n.$ Then, by Proposition~\ref{wcb-prop}, we have a UCP map on $\phi:\cl U_n \to \cl A$ sending $E_{1,2,i}$ to $x_i/2w$ .  Since $B(H)$ is injective we can extend this map to a UCP map from $\oplus_{i=1}^n M_2$ to $B(H)$, which we still denote by $\phi$.

By setting $P_i = \phi(E_{1,1,i})$, $Q_i = \phi(E_{2,2,i})$, and applying Choi's theorem, we obtain that the operator matrix
\begin{align}\label{Eq:dilation isometry-I}
    \begin{pmatrix} P_i & x_i/2w \\x_i^*/2w & Q_i \end{pmatrix}
\end{align} 
is a positive operator.  Since $\phi$ is unital we have that
\begin{align}\label{Eq:sum to 1}
\sum_{i=1}^n P_i + \sum_{i=1}^n Q_i = I_{H}.  
\end{align} 
If $x_i$ is an isometry, conjugate this matrix by
\[ \begin{pmatrix} -x_i^* & 0\\ 0 & I_{H} \end{pmatrix}\] and take the sum of all the entries, to conclude that
\[x_i^*P_ix_i +Q_i - I_{H}/w \ge 0.\]
Now, since $\psi$ is a hyperstate, we get
\begin{align*} 
\psi(x_i^* P_i x_i) & = \psi(P_ix_ix_i^*) \\
&= \psi((x_ix_i^*)^{1/2} P_i (x_ix_i^*)^{1/2}) \\
& \le \psi((x_ix_i^*)^{1/2} P_i (x_ix_i^*)^{1/2}) + \psi( (1-x_ix_i^*)^{1/2} P_i (1 - x_ix_i^*)^{1/2}) \\
& = \psi(P_i). 
\end{align*}
Hence, $\psi(P_i+Q_i) \ge 1/w$.
Similarly, if $x_i$ is a coisometry, we conjugate the matrix \eqref{Eq:dilation isometry-I} by 
\[ \begin{pmatrix} I_{H} & 0 \\ 0 & -x_i \end{pmatrix},\]
and in the same manner, conclude that $\psi(P_i+ Q_i) \ge 1/w$.
Thus, by combining all of these with \eqref{Eq:sum to 1}, we obtain 
\[ 1 = \sum_{i=1}^n \psi(P_i+Q_i) \ge n/w > 1.\]
This contradiction shows that $w_{cb}(x_1,...,x_n) =n$ for every $n$ and every such $n$-tuple.
\end{proof}

We now combine the above results into the following theorem which summarizes the connections between the existence of amenable traces and this joint cb-numerical radius. We recall that for unital C*-algebra $\cl A\subseteq B(H)$ and $x\in \cl A$, $x^{op}$ denote the corresponding element in the opposite C*-algebra $\cl A^{op}$. A concrete representation of $\cl A^{op}$ into $B(\overline{H})$ is given by
\begin{align}\label{Eq:Concrete Rep of opposite algebra}
    A^{op}\ni x^{op}\rightsquigarrow  x^*\in B(\overline{H}).
\end{align}

\begin{thm}\label{T:amne trace vs FJNR-all equivalent conditions}
Let $\cl A$ be a unital C*-algebra. Then the following are equivalent:
\begin{enumerate}
\item $\cl A$ has an amenable trace,
\item for every $n$-tuple of unitaries $(u_1,\ldots,u_n)$ in $\cl A$ and for every $n \in \bb N$ we have that $w_{cb}(u_1,\ldots,u_n) =n$,
\item for every $n$-tuple $(x_1,\ldots,x_n)$ in $\cl A$ such that each $x_i$ is either an isometry or coisometry, we have that $w_{cb}(x_1,\ldots,x_n) =n$,
\item for every $n$-tuple of unitaries $(u_1,\ldots,u_n)$ in $\cl A$ and for every $n \in \bb N$ we have that
$w(\sum_{i=1}^n u_i \otimes (u_i^*)^{op}) =n$ as an element of $\cl A \otimes_{min} \cl A^{op}$,
\item for every $n$-tuple $(x_1,\ldots,x_n)$ in $\cl A$ such that each $x_i$ is either an isometry or coisometry, we have that
$w\left(\sum_{i=1}^n x_i \otimes (x_i^*)^{op}\right)=n$ as an element of $\cl A \otimes_{min} \cl A^{op}$.
\end{enumerate}
Moreover, if $\cl A$ is generated by $(x_1,\ldots,x_n)$ where each $x_i$ is either an isometry or coisometry, then $\cl A$ has an amenable trace if and only if $w_{cb}(x_1,\ldots,x_n) =n$.
\end{thm}
\begin{proof} 

By Theorem \ref{T:amen trace imply max FJNR}, (1) implies (3). Clearly, (5) implies (4), (4) implies (2), and (3) implies (2).
So it remains to show that (2) implies (1) and (1) implies (5).

Suppose that (2) holds. Let $\cl U(\cl A)$ be the set of all unitaries in $\cl A$, and let $\cl F$ be the directed set of all finite subsets of $\cl U(\cl A)$. 
For every $F\in \cl F$, let $\cl A_F$ be the unital C*-subalgebra of $\cl A$ generated by $F$. By our assumption and Proposition \ref{P:amenable trace gen by unitaries}, $\cl A_F$ has an amenable trace so that, by considering a concrete representation $\cl A\subseteq B(H)$, $\cl A_F$ has a hyperstate $\tau_F$ on $B(H)$. This implies that for every $u\in \cl U(\cl A)$
\begin{align}
   \text{weak}^*-\lim_{F\to \infty} u\cdot \tau_F-\tau_F \cdot u=0 ,
\end{align}
where $u \cdot \tau_F(x) = \tau_F(xu)$ and $\tau_F \cdot u(x) = \tau_F(ux)$,
and
where we are viewing $(\cl F, \subseteq)$ as the standard net on $\cl F$. We can then conclude that if we let $\tau$ to be a weak$^*$-cluster point of 
$\{\tau_F: F\in \cl F \}$ in $B(H)^*$, then $\tau$ is a hyperstate on $\cl A$.  

Now suppose that (1) holds. Let $\tau$ be an amenable trace of $\cl A$. Then, by Kirchberg's theorem (Theorem \ref{T:Charcrization amenable trace-Kirchberg}), there is a state
$\phi: \cl A \otimes_{min} \cl A^{op} \to \bb C$ given by 
$$\phi(a \otimes b^{op}) = \tau(ab)\ \ \ \ (a,b\in \cl A).$$
Hence, for every $n$-tuple $(x_1,\ldots,x_n)$ in $\cl A$ such that each $x_i$ is either an isometry or coisometry, we have 
\[ \phi\left(\sum_{i=1}^n x_i \otimes (x_i^*)^{op}\right) = \sum_{i=1}^n \tau(x_ix_i^*) =n,\]
since $\tau(x_ix_i^*) = \tau(x_i^*x_i) = \tau(1)=1$. Thus (5) holds.

Finally, we state that the last statement of the theorem follows from Proposition \ref{P:amenable trace gen by unitaries}  and Theorem \ref{T:amen trace imply max FJNR}. 
\end{proof}

Combining the above result with Connes' deep work on injective factors, we get the following numerical radius results for a finite von Neumann algebra.
\begin{cor}
Let $(\cl M, \tau)$ be a finite von Neumann algebra with a normalized trace $\tau$. Then the following statements are equivalent
\begin{enumerate}
\item $\cl M$ is an injective factor.
\item $\tau$ is amenable.
\item For every $n \in \bb N$ and every $n$-tuple of  unitary operators $(u_1,\ldots,u_n)$ in $\cl M$, we have 
$\|\sum_{i=1}^n \bar{u_i}\otimes u_i\|_{\min}=n.$
\item  For every $n \in \bb N$ and every $n$-tuple of  unitary operators $(u_1,\ldots,u_n)$ in $\cl M$, we have 
$w_{cb}(u_1,\ldots,u_n) =n$.
\item For every $n \in \bb N$ and every $n$-tuple of unitary operators $(u_1, \dots, u_n)$ in $\cl M$, we have
$w(\sum_{i=1}^n (u_i^*)^{op} \otimes u_i) =n$ in $\cl M^{op} \otimes_{min} \cl M$.
\end{enumerate}
\end{cor}
\begin{proof}
The equivalence from (1)-(3) is due to Connes \cite{Connes-76}. The equivalence between (2) and (4)-(5) follows from Theorem \ref{T:amne trace vs FJNR-all equivalent conditions}.
\end{proof}
In the above results we have seen the consequences of $w_{cb}(u_1,...,u_n)$ attaining its upper bound of $n$. Now we establish a lower bound.

\begin{prop}\label{P:min value-sum of unitary tensors}
Let $(u_1,\ldots,u_n)$ be an n-tuple of unitaries on a Hilbert space $H$. Then 
\begin{align}
\sqrt{2n-1}\leq w\left(\sum_{i=1}^n u_i \otimes (u_i^*)^{op}\right)\leq w_{cb}(u_1,\ldots,u_n) .
\end{align}
Moreover, if the unital C*-algebra $\cl A$ generated by $u_1,\ldots,u_n$ has a faithful trace $\tau$, then the following are equivalent:
\begin{enumerate}
\item $w_{cb}(u_1,\ldots,u_n)=\sqrt{2n-1}$; 
\item $w \left(\sum_{i=1}^n u_i \otimes (u_i^*)^{op}\right)=\sqrt{2n-1}$;
\item the reduced C*-algebra, $C^*_{\lambda}(\mathbb{F}_n)$, embeds into $\cl A$ taking $s_i$ into $u_i$, where $\{s_1,\ldots,s_n\}$ are the standard generators of $\mathbb{F}_n$.
\end{enumerate}
\end{prop}

\begin{proof}
Let us first recall that the numerical range of a self-adjoint operator is the same as its operator norm. Thus we have
\begin{align*}
2 w_{cb}(u_1,\ldots,u_n) &\geq 2w\left(\sum_{i=1}^n u_i \otimes (u_i^*)^{op}\right)\\
& \geq w\left(\sum_{i=1}^n u_i \otimes (u_i^*)^{op}+u_i^* \otimes (u_i)^{op}\right)\\ 
& =\left\|\sum_{i=1}^n u_i \otimes (u_i^*)^{op}+u_i^* \otimes (u_i)^{op}\right\|\\ 
&\geq 2\sqrt{2n-1},
\end{align*}
where the last inequality follows from \cite[Theorem 1]{Pisier-1997}. This proves the inequality. Furthermore, it is shown in \cite[Theorem 1.1]{Cad-Collins} that, when $\cl A$
has a faithful trace $\tau$, then 
\begin{align}\label{Eq:min value-sum of unitary tensors}
\left\|\sum_{i=1}^n u_i \otimes (u_i^*)^{op}+u_i^* \otimes (u_i)^{op}\right\|=2\sqrt{2n-1}
\end{align} 
if and only if the group von Neumann algebra $L(\mathbb{F}_n)$ embeds into the finite von Neumann algebras generated by GNS construction of $\tau$ mapping each $s_i$ to $u_i$. This, together with \cite[Proposition 2.4]{DPR}, yields the equivalent of (1)-(3).  
\end{proof}

\begin{rem} 
(i) In the preceding proposition, the assumption of the existence of a faithful tracial state on $\cl A$ is necessary and cannot be dropped. Indeed, it is pointed out in \cite[p.3]{Cad-Collins} that, based on a counterexample that appeared in Franz Lehner's Ph.D. thesis (\cite[p.51]{Leh}), if $\cl A$ does not have a faithful trace, then the relations \eqref{Eq:min value-sum of unitary tensors} does not necessarily imply that $u_1,\ldots,u_n$ are free Haar unitaries. In such cases, $C^*_{\lambda}(\bb F_n)$ need not embed into $\cl A$ so that either (1) or (2) may not imply (3) in Proposition \ref{P:min value-sum of unitary tensors}.

(ii) We have seen in both Theorem \ref{T:amne trace vs FJNR-all equivalent conditions} and Proposition \ref{P:min value-sum of unitary tensors} that, for any $n$-tuple of unitaries $u_1,\ldots,u_n$ on a Hilbert space $H$, the quantities $w\left( \sum_{i=1}^n u_i \otimes (u_i^*)^{op}\right)$ and $w_{cb}(u_1,\ldots,u_n)$ obtain their minimal or maximal value at the same time. However, in Proposition \ref{P:failure FJNR equal tensor with op}, we give an example of an $n$-tuple of unitaries $u_1,\ldots,u_n$ such that
\[ w\left( \sum_{i=1}^n u_i \otimes (u_i^*)^{op}\right) < w_{cb}(u_1,\ldots,u_n),\]
so that these two quantities need not always coincide. On the other hand, we will show in Proposition \ref{P:FJNR vs tensor with op} that there is a general relation between these two quantities.   
\end{rem}

It would be interesting to compute the exact value, or at least approximate the joint numerical radius of $n$-tuples of operators when they do not obtain the maximal or minimal values. We finish this section by computing one such case for the Cuntz algebras.

\begin{prop}
Let $n\in \mathbb{N}$, and let $x_1,\ldots, x_n$ be isometries that generate the  Cuntz algebra $\cl O_n$. Then $w_{cb}(x_1,\ldots,x_n)=\sqrt{n}$.
\end{prop}	

\begin{proof}
It is well-known that Cuntz isometries satisfy $x_i^*x_j=0 \ \text{if} \ \ 1\leq i\neq j\leq n$. 
Thus, for any unitaries $U_1,\dots, U_n \in M_k$ and a state $\psi$ on $\cl O_n\otimes M_k$, we have (by the Cauchy-Schwartz inequality)
$$\left|\psi\left(\sum_{i=1}^n x_i\otimes U_i\right)\right|\leq  \psi\left(\sum_{i=1}^n x_i^*x_i\otimes 1\right)^{1/2}=\psi(n)^{1/2}=\sqrt{n}.$$
Hence, it follows from Definition \ref{D:free joint numerical radius} that  $w_{cb}(x_1,\ldots,x_n)\leq \sqrt{n}$.

On the other hand, we can view $\cl O_n$ as the unital C*-algebra in $B(L^2(0,1))$ generated by Cuntz isometries ($f\in L^2(0,1)$, $i=1,\ldots,n$)
$$
\begin{array}{cc}
x_i f(x)= \bigg\{  
    \begin{array}{cc}
      \sqrt{n} f(nx + 1 - i)\  &  (i-1)/n \leq x \leq i/n \\
        0 & \text{otherwise}
    \end{array}
\end{array}
$$
Then, by putting $g=1_{(0,1)}$, we get
$$\langle (x_1 + \dots + x_n) g,g \rangle = \sum_{i=1}^n \frac{\sqrt{n}}{n} = \sqrt{n}.$$
Hence $\sqrt{n}=w\left(\sum_{i=1}^n x_i\right)\leq w_{cb}(x_1,\ldots,x_n).$ This completes the proof. 
\end{proof}

\section{C*-algebras generated by partial isometries}

There is an extensive literature on graph C*-algebras and these C*-algebras are generated by families of partial isometries. In this section we turn our attention to obtaining results giving necessary and sufficient conditions for algebras generated by partial isometries to possess an amenable trace. We first need some preliminary results that we will use later. 


\begin{prop}\label{P:FJNR vs tensor with op}
Let $(x_1,\ldots,x_n)$ be an n-tuple of contractions on a Hilbert space $H$. Then 
\begin{align}\label{Eq:FJNR vs tensor with op}
w\left(\sum_{i=1}^n x_i \otimes (x_i^*)^{op}\right)\leq w_{cb}(x_1,\ldots,x_n) \leq \sqrt{n} w\left(\sum_{i=1}^n x_i \otimes (x_i^*)^{op}\right)^{1/2}.
\end{align}
\end{prop}

\begin{proof}
It follows from Definition \ref{D:free joint numerical radius}, by putting $U_i=(x_i)^{op}$, $i=1,\dots,n$, that we have
$$w\left(\sum_{i=1}^n x_i \otimes (x_i^*)^{op}\right)\leq w_{cb}(x_1,\ldots,x_n).$$
Now take contractions $U_1,\dots, U_n \in M_k$ and a unit vector $v=\sum_g v_g \otimes \xi_g \in \mathbb{C}^k \otimes H$ with $\{v_g\}$ being orthonormal in  $\mathbb{C}^k$ (so that $\sum_g \|\xi_g\|^2=1$), we have
\begin{align*}
 \big| \big\langle \big( \sum_i U_i \otimes x_i \big) v , v \big\rangle \big|^2
 &\le \left(\sum_{i,g,h} |\langle U_i v_g , v_h \rangle | |\langle x_i\xi_g , \xi_h \rangle|\right)^2 \\
  &\le n\sum_{i,g,h} |\langle x_i\xi_g , \xi_h \rangle|^2 \ \ \ \ (\text{by Holder's inequality})\\ 
  &= n\sum_{i,g,h} \langle x_i\xi_g , \xi_h \rangle \overline{\langle x_i\xi_g , \xi_h \rangle}\\
 & = n\sum_{i,g,h} \langle x_i\xi_g , \xi_h \rangle \langle (x_i^*)^{op}\xi_g , \xi_h \rangle \ \ \ \ (\text{by} \ \eqref{Eq:Concrete Rep of opposite algebra})\\
 & = n \sum_{i,g,h}  \langle (x_i \otimes (x_i^*)^{op})(\xi_g\otimes \xi_g) , \xi_h \otimes \xi_h \rangle \\
 &= n\langle (\sum_i x_i \otimes (x_i^*)^{op})(u) , u \rangle \\
\end{align*}
where $u = \sum_g \xi_g\otimes \xi_g \in H\otimes \overline{H}$. Since
$$\|u\|_{H\otimes \bar{H}}\leq \sum_g \|\xi_g\|^2=1,$$
we have
 \[
w_{cb}(x_1,\ldots,x_n) \leq \sqrt{n} w\left(\sum_{i=1}^n x_i \otimes (x_i^*)^{op}\right)^{1/2}.
\]
\end{proof}

\begin{rem}
Let $(x_1,\ldots,x_n)$ be an $n$-tuple in a unital C$^*$-algebra $\cl A$, where each $x_i$ is either an isometry or coisometry. Motivated by the relation \eqref{Eq:FJNR vs tensor with op} and Theorem \ref{T:amne trace vs FJNR-all equivalent conditions}, one might wonder whether we always have
\begin{align}\label{Eq:FJNR equal tensor with op}
w_{cb}(x_1,\ldots,x_n) = w\left(\sum_{i=1}^n x_i \otimes (x_i^*)^{op}\right),    
\end{align}
where the latter is taken in $\cl A \otimes_{min} \cl A^{op}$? This holds when $\cl A$ has an amenable trace (Theorem \ref{T:amne trace vs FJNR-all equivalent conditions}). It also holds when $x_i=\lambda(g_i)$, $i=1,\ldots,n$, where $\{g_i:i=1,\ldots,n\}$ is a subset of a discrete group $G$. Indeed,  
by Fell's absorption principle, the representation $g\rightarrow \lambda(g)\otimes (\lambda(g)^*)^{op}$ is unitarily equivalent to $g\rightarrow \lambda(g)\otimes 1_G$, where ${1_G}$ is the trivial representation. Then, since the joint numerical radius is a unitary invariant,
\[w\left(\sum_{i=1}^n  \lambda(g)\otimes (\lambda(g)^*)^{op}\right)=w\left(\sum_{i=1}^n  \lambda(g_i)\otimes 1\right)=w\left(\sum_{i=1}^n  \lambda(g_i)\right)=w_{cb}(\lambda(g_1), \ldots, \lambda(g_n)).\]
The last equality follows from \cite[Proposition 2.2]{DPR}. This shows that \eqref{Eq:FJNR equal tensor with op} holds when we take the left regular representation. However, we will show in Proposition \ref{P:failure FJNR equal tensor with op} that the equality \eqref{Eq:FJNR equal tensor with op} fails for a general representation.
\end{rem}

\begin{prop}\label{P:partial isom-compare JFNR with NR on min tensor}
For every $n \in \bb N$, let $(x_1,\ldots,x_n)$ be $n$-tuple of operators such that each $x_i\in B(H)$ is a partial isometry. Then 
\begin{align}
w_{cb}(x_1,1-x_1x_1^*,x_1^*,1-x_1^*x_1,\ldots,x_n, 1-x_nx_n^*,x_n^*,1-x^*_nx_n) \leq 2n.
\end{align}
\end{prop}

\begin{proof}
     We first note that by the Cauchy-Schwartz inequality, for every $m\in \mathbb{N}$ and $y_1,\ldots,y_m \in B(H)$ and any state $\psi$ on $B(H)$, we have
  \begin{align}\label{Eq:CS-inequality}
    \left|\psi\left(\sum_{i=1}^{m} y_i \right)\right|\leq \sqrt{m}\, \psi\left(\sum_{i=1}^{m} y_i^*y_i \right)^{1/2}.
  \end{align}
Now let $\{ U_{i1},U_{i2},U_{i3},U_{i4} \}$, $i=1,\ldots,n$ be unitary matrices in $M_k$. Then, by letting
$$y_i=x_i\otimes U_{i1}+(1-x_ix^*_i)\otimes U_{i2} \ \ , \ \ y_{i+n}=x_i^*\otimes U_{i3}+(1-x_i^*x_i)\otimes U_{i4},$$
and using the fact that $x_i=x_ix_i^*x_i$ and $x_i^*=x_i^*x_ix_i^*$, we have
$$y_i^*y_i=(x_i^*x_i+1-x_ix^*_i)\otimes I_k \ \ , \ \ y_{i+n}^*y_{i+n}=(x_ix_i^*+1-x_i^*x_i)\otimes I_k.$$
Hence, by \eqref{Eq:CS-inequality},
\begin{align*}
    & \left|\psi\left(\sum_{i=1}^{n} x_i\otimes U_{i1}+(1-x_ix^*_i)\otimes U_{i2}+x_i^*\otimes U_{i3}+(1-x_i^*x_i)\otimes U_{i4} \right)\right| \\
   & \leq \sqrt{2n} \psi\left(\sum_{i=1}^{n} (x_i^*x_i+1-x_ix^*_i)\otimes I_k+(x_ix_i^*+1-x_i^*x_i)\otimes I_k \right)^{1/2}\\
   & = \sqrt{2n} \psi\left(\sum_{i=1}^{n} (x_i^*x_i+1-x_ix^*_i+x_ix_i^*+1-x_i^*x_i)\otimes I_k \right)^{1/2} \\
    & = \sqrt{2n} \left(\sum_{i=1}^{n} 2 \right)^{1/2} \\
    &=2n.
\end{align*}
\end{proof}

We can now present the main result of this section that characterizes when C*-algebras generated by partial isometries have amenable traces in terms of a numerical radius.  

\begin{thm}\label{T:amenatble trace-C alg generated by partical isometries}
For every $n \in \bb N$, let $(x_1,\ldots,x_n)$ be an $n$-tuple of operators such that each $x_i\in B(H)$ is a partial isometry. Then 
the unital C*-algebra generated by the set $\{x_1, \ldots, x_n\}$ has an amenable trace if and only if 
\begin{align}\label{Eq:NR op alg-partial isometry}
w\left( \sum_{i=1}^n x_i\otimes (x_i^*)^{op} +(1-x_ix^*_i)\otimes (1-x_ix^*_i)^{op}+x_i^* \otimes (x_i)^{op}+(1-x_i^*x_i)\otimes (1-x_i^*x_i)^{op} \right)=2n.
\end{align}
\end{thm}

\begin{proof}
Let $\cl A$ to be the unital C*-algebra generated by $\{x_i\}$. 
Suppose that $A$ has an amenable trace and $\tau$ is an amenable on $\cl A$. It follows from Propositions \ref{P:FJNR vs tensor with op} and \ref{P:partial isom-compare JFNR with NR on min tensor} that the quantity in the left hand-side of \eqref{Eq:NR op alg-partial isometry} is less than or equal to $2n$. On the other hand, by Theorem \ref{T:Charcrization amenable trace-Kirchberg}, there is a state
$\phi: \cl A \otimes_{min} \cl A^{op} \to \bb C$ given by 
$$\phi(a \otimes b^{op}) = \tau(ab)\ \ \ \ (a,b\in \cl A).$$ 
In particular,
\begin{align*}
&\phi\left( \sum_{i=1}^n x_i\otimes (x_i^*)^{op} +(1-x_ix^*_i)\otimes (1-x_ix^*_i)^{op}+x_i^* \otimes (x_i)^{op}+(1-x_i^*x_i)\otimes (1-x_i^*x_i)^{op} \right)= \\ 
& \tau \left( \sum_{i=1}^n x_i x_i^*+1-x_ix^*_i+x_i^* x_i+1-x_i^*x_i\right)=2n.
\end{align*}
 
Conversely, suppose that \eqref{Eq:NR op alg-partial isometry} holds. Then there is a state $\phi$ on $\cl A\otimes_{min} \cl A^{op}$ such that 
 \begin{align*}
\left|\phi\left( \sum_{i=1}^n x_i\otimes (x_i^*)^{op} +(1-x_ix^*_i)\otimes (1-x_ix^*_i)^{op}+x_i^* \otimes (x_i)^{op}+(1-x_i^*x_i)\otimes (1-x_i^*x_i)^{op} \right)\right|=2n 
\end{align*}
Hence, if we put
\begin{align*}
    P_i=x^*_ix_i \ \ , \ \ Q_i=x_ix_i^*,
\end{align*}
then, by \eqref{Eq:CS-inequality}, with 
$$y_i=x_i\otimes (x_i^*)^{op} +(1-x_ix^*_i)\otimes (1-x_ix^*_i)^{op}, \ \ i=1,\ldots,n $$ and 
$$y_{i+n}=x_i^* \otimes (x_i)^{op}+(1-x_i^*x_i)\otimes (1-x_i^*x_i)^{op}, \ \ i=1,\ldots,n $$ we obtain
 \begin{align*}
2n& = \left|\phi\left( \sum_{i=1}^n x_i\otimes (x_i^*)^{op} +(1-x_ix^*_i)\otimes (1-x_ix^*_i)^{op}+x_i^* \otimes (x_i)^{op}+(1-x_i^*x_i)\otimes (1-x_i^*x_i)^{op} \right)\right|\\
&\leq \sqrt{2n}\,\phi\left( \sum_{i=1}^n P_i\otimes P_i^{op} +(1-P_i)\otimes (1-P_i)^{op}+Q_i \otimes Q_i^{op}+(1-Q_i)\otimes (1-Q_i)^{op} \right)^{1/2}\\
&\leq 2n, 
\end{align*}
where the last equality follows since each $P_i\otimes P_i^{op} +(1-P_i)\otimes (1-P_i)^{op}$ and $Q_i \otimes Q_i^{op}+(1-Q_i)\otimes (1-Q_i)^{op}$ is a projection. Thus, for $i=1,\ldots,n$, we must have
\begin{align}\label{Eq:}
    \phi\left(P_i\otimes P_i^{op} +(1-P_i)\otimes (1-P_i)^{op}\right)=\phi\left(Q_i \otimes Q_i^{op}+(1-Q_i)\otimes (1-Q_i)^{op}\right)=2
\end{align}
so that
\begin{align}\label{Eq:state vanishes on projections-II}
    \phi\left(P_i\otimes 1+1\otimes P_i^{op}-2P_i\otimes P_i^{op}\right)= \phi\left(Q_i\otimes 1+1\otimes Q_i^{op}-2Q_i\otimes Q_i^{op}\right)=0.
\end{align}
It is a well-known fact, due to $\phi$ being an state, that the set $\{a\in \cl A\otimes_{min} \cl A^{op} : \phi(a^*a)=0 \}$ is a left ideal of $\cl A\otimes_{min} \cl A^{op}$ and is contained in the kernel of $\phi$. We will use this fact frequently throughout our argument. For each $i=1,\ldots,n$, 
we have
$$(P_i\otimes 1-1\otimes P_i^{op})^*(P_i\otimes 1-1\otimes P_i^{op})=P_i\otimes 1+1\otimes P_i^{op}-2P_i\otimes P_i^{op}.$$
Hence, by \eqref{Eq:state vanishes on projections-II}, $\phi(P_i\otimes 1-1\otimes P_i^{op})=0$, so that by applying once again \eqref{Eq:state vanishes on projections-II}, we get
$$\phi(P_i\otimes P_i^{op})=2\phi(P_i\otimes 1)=2\phi(1\otimes P_i^{op}).$$
Similarly, 
$$\phi(Q_i\otimes Q_i^{op})=2\phi(Q_i\otimes 1)=2\phi(1\otimes Q_i^{op}).$$
By combining the above two relations with how we choose $\phi$, we get
 \begin{align}\label{Eq:state vanishes on projections-III}
2n=\left|2n-\phi\left( \sum_{i=1}^n Q_i\otimes Q_i^{op}-x_i\otimes (x_i^*)^{op} -x_i^* \otimes (x_i)^{op}+P_i\otimes P_i^{op} \right)\right|.
\end{align}
However, each $Q_i\otimes Q_i^*-x_i\otimes (x_i^*)^{op} -x_i^* \otimes (x_i)^{op}+P_i\otimes P_i^{op}$ is a positive operator as
$$(x_i^*\otimes x_i^{op} -P_i\otimes P_i^{op})^*(x_i^* \otimes x_i^{op}-P_i\otimes P_i^{op})=Q_i\otimes Q_i^*-x_i\otimes (x_i^*)^{op} -x_i^* \otimes (x_i)^{op}+P_i\otimes P_i^{op}.$$
Therefore, by \eqref{Eq:state vanishes on projections-III}, for each $i=1,\ldots,n$,
$$\phi\left( Q_i\otimes Q_i^{op}-x_i\otimes (x_i^*)^{op} -x_i^* \otimes (x_i)^{op}+P_i\otimes P_i^{op} \right)=0.$$
This implies that $\phi\left(x_i^* \otimes (x_i)^{op}-P_i\otimes P_i^{op} \right)=0$ so that
\begin{align}\label{Eq:state vanishes on projections-IV}
\phi\left(x_i^* \otimes (x_i)^{op})=\phi(P_i\otimes P_i^{op} \right)=\phi\left(P_i\otimes 1\right)=\phi\left(1\otimes P_i^{op} \right).
\end{align}
Similarly, we have
\begin{align}\label{Eq:state vanishes on projections-V}
\phi\left(x_i \otimes (x_i^*)^{op})=\phi(Q_i\otimes Q_i^{op} \right)=\phi\left(Q_i\otimes 1\right)=\phi\left(1\otimes Q_i^{op} \right).
\end{align}
Finally, from the relation
$$(x_i\otimes 1 -1\otimes x_i^{op})^*(x_i \otimes 1-1\otimes x_i^{op})=P_i\otimes 1-x_i^*\otimes (x_i)^{op} -x_i \otimes (x_i^*)^{op}+1\otimes Q_i^{op},$$
and equations \eqref{Eq:state vanishes on projections-IV} and \eqref{Eq:state vanishes on projections-V}, it follows that the left ideal in $\cl A\otimes_{min} \cl A^{op}$ generated by $x_i\otimes 1-1\otimes (x_i)^{op}$ lies in $\ker \phi$. In particular,
\begin{align*}
    \phi(ax_i\otimes c^{op})=\phi(a\otimes (x_ic)^{op}) \ \ \ (a,c \in \cl A, i=1,\ldots,n).
\end{align*}
Similarly, we can show that 
\begin{align*}
    \phi(ax^*_i\otimes c^{op})=\phi(a\otimes (x^*_ic)^{op}) \ \ \ (a,c \in \cl A, i=1,\ldots,n).
\end{align*}
Since $\cl A$ is generated by $\{x_i\}$ and $\{x_i^*\}$, it is straightforward to verify that the preceding two relations will imply 
\begin{align*}
    \phi(ab\otimes c^{op})=\phi(a\otimes (bc)^{op}) \ \ \ (a,b,c \in \cl A).
\end{align*}
Putting $c=1$ and defining the stat $\tau$ on $\cl A$ by 
$$\tau(a)=\phi(a\otimes 1) \ \ \ (a\in \cl A),$$
we will obtain
$$\tau(ab)=\phi(ab\otimes 1)=\phi(a\otimes b^{op}) \ \ (a,b\in \cl A).$$
It is automatic that any such state $\tau$ that satisfies the above relation must be a trace. This follows since, for every self-adjoint elements $a,b\in \cl A$, $\tau(ab)=\phi(a\otimes b^{op})\in \mathbb{R}$. Hence
$$\tau(ab)=\overline{\tau(ab)}=\tau((ab)^*)=\tau(ba).$$
The final claim follows from the density of self-adjoint elements in $\cl A$. 

  \end{proof}

We have the following corollary that characterizes, in terms of a numerical radius, when a unital C*-algebra generated by projections has an amenable trace. 

\begin{cor}
For every $n \in \bb N$, let $(P_1,\ldots,P_n)$ be an $n$-tuple of projections in $B(H)$. Then 
the unital C*-algebra generated by $\{P_i\}$ has an amenable trace if and only if
$$w\left( \sum_{i=1}^n P_i\otimes P_i^{op} +(1-P_i)\otimes (1-P_i)^{op} \right)=n.$$

\end{cor}

\begin{proof}
The result follows from Theorem \ref{T:amenatble trace-C alg generated by partical isometries} by putting $x_i=P_i$. 
\end{proof}

\begin{rem}
It would be interesting to see whether the conditions in Theorem~\ref{T:amenatble trace-C alg generated by partical isometries} would be also equivalent to
\begin{align*}
w_{cb}(x_1,1-x_1x_1^*,x_1^*,1-x_1^*x_1,\ldots,x_n, 1-x_nx_n^*,x_n^*,1-x^*_nx_n)=2n?
\end{align*}
Of course, having an amenable trace will imply this condition; this is straighforward as can be seen by  combining Propositions \ref{P:FJNR vs tensor with op}, \ref{P:partial isom-compare JFNR with NR on min tensor} and Theorem \ref{T:amenatble trace-C alg generated by partical isometries}. However,  we do not see how to obtain the converse. 
\end{rem}

\begin{rem} For graph C*-algebras it would be interesting to try and give formulas for the numerical radius appearing in 
Theorem~\ref{T:amenatble trace-C alg generated by partical isometries} for the generating partial isometries from properties of the defining graph.
\end{rem} 

\section{Amenable Traces and the local lifting property (LLP)}

In this section we look at the relation between various lifting results and the existence of amenable traces. We revisit some of the ideas in \cite{DPR}. First we need a slight reformulation of the result \ref{T:Charcrization amenable trace-Kirchberg}.

\begin{thm}\label{T:amen tarce-LLP for f.d. operator system}
    Let $\cl A$ be a separable C$^*$-algebra, and let $\tau$ be a tracial state on $\cl A$. Then $\tau$ is amenable if and only if 
        there is a $*$-homomorphsim 
        $\pi: \cl A\rightarrow \cl R^{\omega}= \ell^{\infty}(\cl R)/I$ with $tr_{\cl R^{\omega}}\circ\pi=\tau$ such that for every  
         $n$ and unitaries $u_1, \ldots, u_n\in \cl A$, the restriction of $\pi$ on the finite dimensional operator system $S=\text{span}\{1, u_1, u_1^*, \ldots, u_n, u_n^*\}$ is liftable into $\ell^{\infty}(\cl R)$.
        \end{thm}

\begin{proof}
The ``only if" part follows from Theorem \ref{T:Charcrization amenable trace-Kirchberg}. Now suppose that there is a $*$-homomorphsim $\pi: \cl A\rightarrow \cl R^{\omega}= \ell^{\infty}(\cl R)/I$ satisfying the hypothesis of the theorem. 
For each finite dimensional operator system $\cl S$ constructed from finitely many unitaries in $\cl A$, let $\Phi_{\cl S}:\cl S\to \ell^\infty(\cl R)$ be the lift of $\pi_{\cl S}$.  Using the fact that $\ell^{\infty}(\cl R)$ is injective, we may extend this to a contractive CP map 
$\Psi_S:\cl A \to \ell^\infty(\cl R)$. Now since $\ell^{\infty}(R)$ is a dual space, we have the contractive CP maps,
$CPC(A, \ell^{\infty}(\cl R))$ is compact in Arveson's BW-topology.  So we may take a BW-limit point of the net $\Psi_S$ and that will be a lift defined on all of $\cl A$.    
\end{proof}

We have now the following corollary that relates existence of amenable traces with the local lifting property. 
\begin{cor}\label{C:LLP imply amen trace or non Connes embed}
    Let $\cl A$ be a unital C$^*$-algebra with LLP. Then either $\cl A$ has an amenable trace or there is no unital $*$-homomorphism $\pi:\cl A\rightarrow \cl R{^\omega}$.
\end{cor}

\begin{proof}
Suppose that $\cl A$ has no amenable trace. By Theorem \ref{T:amne trace vs FJNR-all equivalent conditions}, there exists an $n\in \mathbb{N}$ and unitaries $u_1, \ldots, u_n$ in $\cl A$ such that the (separable) C*-algebras $\cl A_{\cl S}$ generated by $u_1, \ldots, u_n$ (in $\cl A$) has no amenable trace. Hence if there is $*$-homomorphism $\pi:\cl A\rightarrow \cl R{^\omega}$, then, by Theorem \ref{T:amen tarce-LLP for f.d. operator system}, the restriction of $\pi$ on 
$S=\text{span}\{1, u_1, u_1^*, \ldots, u_n, u_n^*\}$ is not liftable to $\ell^{\infty}(R)$ which contradicts our hypothesis that $\cl A$ has LLP. 
\end{proof}

The following result also implies the above corollary and does not rely on the characterization of amenable traces given in Theorem \ref{T:amen tarce-LLP for f.d. operator system}. It also removes the requirement that the ``numerator" algebra be a dual space that was needed to use the BW-topology.

\begin{thm} Let $\cl B$ be an injective unital C*-algebra, $J \subseteq \cl B$ a 2-sided ideal and assume that $\cl B/J$ has a tracial state. If $\cl A$ has no amenable trace, then any unital $*$-homomorphism from $\cl A$ into $\cl B/J$ fails to satisfy the LLP.
\end{thm}
\begin{proof} Assume that a unital $*$-homomorphism $\pi: \cl A \to \cl B/J$ exists that satisfies the LLP.
Since $\cl A$ has no amenable trace, then there exists an $n\in \mathbb{N}$ and unitaries $u_1, \ldots, u_n$ in $\cl A$ with $w_{cb}(u_1, \ldots, u_n)<n$ (by Theorem \ref{T:amne trace vs FJNR-all equivalent conditions}). Then 
\[w:= w_{cb}(\pi(u_1),\ldots, \pi(u_n)) \le w_{cb}(u_1,\ldots,u_n) <n.\]
Now look at the operator system $S=\text{span}\{1, u_1, u_1^*, \ldots, u_n, u_n^*\}$. One can proceed similarly as in \cite[Theorem 3.1]{DPR} to get a contradiction. More precisely, one considers the UCP map
\[ \Phi: \cl U_n \to S, \,\, \Phi(E_{1,2,i}) = \pi(u_i)/2w,\]
where $\cl U_n$ is the operator system in $\oplus_{i=1}^n M_2$ introduced in Section \ref{S:Operator systems and LLP}.
Next, one uses the injectivity of $\oplus_{i=1}^n M_2$ and LLP property of $\pi$ to obtain a lifting of $S$ and hence of $\Phi$ 
to $\oplus_{i=1}^n M_2$ that passes to the quotient and then considers the Choi matrices of this quotient map. The rest of the argument will follow as in \cite[Theorem 3.1]{DPR} to get a contradiction.
\end{proof}

\begin{cor} If $\cl A$ has no amenable trace, then $\cl A$ is not quasi-diagonal.
\end{cor}
\begin{proof} Take $\cl B = \prod_{n \in \bb N} M_n$ and $J= \oplus_{n \in \bb N} M_n$ and observe that $\cl B/J$ has a tracial state and apply the above result.
\end{proof}  


Throughout the rest of this section, let $G$ be a discrete group. A unitary representation $\pi:G\to B(H)$ is {\it amenable} (in the sense of  Bekka) if $C^*_\pi(G)$, the C*-algebra generated by $\pi(G)$ in $B(H)$, has an amenable trace \cite{Bekka1}. We can state the following result, which characterizes amenable representations in terms of free joint numerical radius. 

\begin{thm}
    Let $G$ be a group, and let $\pi$ a unitary representation of $G$. Then:\\
    $(i)$ $\pi$ is amenable if and only if for every $s_1,\ldots,s_n$ in $G$, we have 
    $w_{cb}(\pi(s_1),\ldots,\pi(s_n))=n$. \\
    $(ii)$ If $G$ is finitely generated and $S=\{s_1,\ldots,s_n\}$ is a generating subset of $G$, then $\pi$ is amenable if and only if 
    $w_{cb}(\pi(s_1),\ldots,\pi(s_n))=n$. 
\end{thm}

\begin{proof}
    This is an immediate consequence of Theorem \ref{T:amne trace vs FJNR-all equivalent conditions} by putting $u_i=\pi(s_i)$, for $i=1,\ldots,n$.
\end{proof}
Next, we write down a theorem that extends one of the main results of \cite[Corollary 3.6]{DPR}

\begin{thm}\label{T:non amen rep-fail LLP}
Let $G$ be a hyperlinear group, and let $\pi$ be a non-amenable representation of $G$ such that the left regular representation $\lambda$ is weakly contained in $\pi$. 
Then $C^*_\pi(G)$  does not have LLP.
\end{thm}
\begin{proof}
As $G$ is hyperlinear, there is an embedding $\iota$ of $C^*_\lambda(G)$ into $\cl R^\omega=\ell^{\infty}(\cl R)/ {\cl I}$. If the continuous homomorphism from $C^*_\pi(G)$ onto $C^*_\lambda(G)$ is denoted by $\phi$, then we have a $*$-homomorphism $\Phi: C^*_\pi(G)\rightarrow \cl R^\omega$ given by
$\Phi=\phi\circ\iota$. Now the result follows from Corollary \ref{C:LLP imply amen trace or non Connes embed}.

\end{proof}

There are many examples that fall in this category. We present some of them below. We recall that a unitary representation of $G$ is \emph{weakly mixing} if it does not contain a finite dimensional subrepresentation. 

\begin{cor} 
Let $G$ be a hyperlinear group with Property (T), and let $\pi:G\to B(H)$ be a weakly mixing unitary representation of $G$. If $\pi$ weakly contains $\lambda$, then $C^*_\pi(G)$ does not have LLP. This holds, in particular, for all weakly mixing representations of higher rank irreducible lattices of noncompact semisimple Lie groups with trivial center, 
\end{cor}

\begin{proof}
 We note that since $G$ has Property (T), then $\pi$ is amenable if and only if it is weakly mixing \cite[Corollary 5.9]{Bekka1}. Thus the result follows from 
 Theorem \ref{T:non amen rep-fail LLP}. The last statement follows since a higher rank irreducible lattice of a noncompact semisimple Lie group with trivial center has property (T) and its weakly mixing representations weakly contain the left regular representation \cite[Corollary D]{BH1}.
\end{proof}

Another class of examples comes from exotic C*-algebras of groups. Let us first recall some terminology. Fix a discrete group $G$. Let $\mathcal{D}$ be an algebraic ideal (not necessary closed) in $\ell^\infty(G)$ that contains $c_{00}(G)$. 
A unitary representation $\pi:G\to B(H)$ is called a $\mathcal{D}$-representation if there is a dense subset $H_0$ of $H$ such that for every $\xi,\eta\in H_0$, the coefficient function $\langle \pi(\cdot)\xi,\eta \rangle$ belongs to $\mathcal{D}$. We let $C^*_{\mathcal{D}}(G)$, called an \emph{exotic} 
$C^*$-algebra of $G$, to be the completion of 
$\ell^1(G)$ with respect to the C*-norm $\|\cdot\|_{\mathcal{D}}$ defined by
$$\|f\|_{\mathcal{D}}=\sup\big\{\|\pi(f)\|: \pi\text{ is an $\mathcal{D}$-representation of }G\big\}.$$
These exotic C*-algebras were first introduced by Brown and Guentner in \cite{BG} and have been widely studied since then (see \cite{SW} and the reference therein). A case of particular interest is when $\mathcal{D}=\ell^p(G)$, for $p\in [1,\infty]$, and $C^*_{\ell^p}(G)$ is called \emph{exotic} 
$\ell^p$-$C^*$-algebra of $G$. We list a few properties that these C*-algebras possess; see \cite{BG} and \cite{SW} for proofs. 

\begin{prop}\label{P:exotic C*-alg properties}
    Let $G$ be a discrete nonamenable group, let $\mathcal{D}$ be an algebraic ideal in $\ell^\infty(G)$ that contains $c_{00}(G)$, and let $p\in [1,\infty]$. Then the following holds:
\begin{enumerate}
 	\item The identity map on $\ell^1(G)$ extends to surjective $*$-homomorphism
	$$ C^*(G)\to C^*_{\mathcal{D}}(G)\to C^*_\lambda(G);$$
	\item $C^*_{\ell^p}(G)=C^*_\lambda(G)$ for each $1\leq p\leq 2$, and $C^*_{\ell^\infty}(G)=C^*(G)$;
    \item If $C^*_{\ell^p}(G)=C^*(G)$ for some $1\leq p<\infty$, then $G$ is amenable;
    \item $G$ has property (T) if and only if $C^*(G)\neq C^*_{\mathcal{D}}(G)$ canonically whenever $\mathcal{D}\neq \ell^\infty(G)$;
    \item For a large class of nonamenable groups, including free groups,  $C^*_{\ell^p}(G)$ are canonically pairwise distinct when $p\in [2,\infty]$;
	    \item If $p\geq 2$, then for every $f\in \ell^1(G)$, we have 
    \begin{align}
        \|f\|_{C^*_{\ell^p}(G)}\leq \|f\|_1^{1-\theta}\|f\|_{C^*_\lambda(G)}^\theta \ \ \text{with} \ \ \theta=\frac{2}{p}.
    \end{align}
\end{enumerate}
\end{prop}

Our result, together with \cite[Proposition 5.2]{RW}, states that these exotic algebras behave more like the reduced C$^*$-algebras when it comes to having an amenable trace and LLP.

\begin{thm}
    Let $G$ be a discrete nonamenable group, and let $\mathcal{D}$ be an algebraic ideal in $\ell^\infty(G)$ that contains $c_{00}(G)$ such that $C^*(G)\neq C^*_{\mathcal{D}}(G)$ canonically. Then $C^*_{\mathcal{D}}(G)$ has no amenable trace. If, in addition, $G$ is hyperlinear, then $C^*_{\mathcal{D}}(G)$ does not have LLP
\end{thm}

\begin{proof}
By \cite[Proposition 5.2]{RW}, $C^*_{\mathcal{D}}(G)$ has no amenable trace. The final statement follows from Theorem \ref{T:non amen rep-fail LLP}. 
\end{proof}

We finish this section with the following proposition to show that the equality \eqref{Eq:FJNR equal tensor with op} fails for many $\ell^p$-representation on free groups.

\begin{prop}\label{P:failure FJNR equal tensor with op}
Let $\mathbb{F}_n$ be the free group on $n$ generators with the standard generating elements $s_1,\dots,s_n$, and let $p\geq 4$. If $n>2^{p+1}$, then there is an $\ell^p$-representation $\pi$ of $\mathbb{F}_n$ such that
\[w\left( \sum_{i=1}^n \pi(s_i) \otimes \bar{\pi}(s_i) \right) <w_{cb}(\pi(s_1),\ldots,\pi(s_n)).\]
\end{prop}

\begin{proof}
    Let $\pi$ be an $\ell^p$-representation on $\mathbb{F}_n$. If the relation \eqref{Eq:FJNR equal tensor with op} always holds, then we must have (as $(\pi^*)^{op}=\bar{\pi}$)
\begin{align}\label{Eq:1}
  w\left(\sum_{i=1}^{n} \pi(s_i)\right)\leq w\left(\sum_{i=1}^n \pi(s_i) \otimes  \bar{\pi}(s_i)\right).
\end{align}
We will show that this does not hold in general. 
We first note that since $\pi\otimes \bar{\pi}$ is a $\ell^{p/2}$-representation, by Proposition \ref{P:exotic C*-alg properties}(5), we have
\begin{align*}
  \|\pi\otimes \bar{\pi}(f)\|\leq\|f\|_1^{1-4/p} \|\lambda(f)\|^{4/p} \ \ \ (f\in \ell^1(\mathbb{F}_n), p\geq 4).
\end{align*}
In particular, 
\begin{align}\label{Eq:2}
w\left(\sum_{i=1}^n \pi(s_i) \otimes \bar{\pi}(s_i)\right)\leq n^{1-4/p}\left\|\sum_{i=1}^n \lambda (s_i) \right\|^{4/p}\leq 2n^{1-4/p}(2n-1)^{2/p}.
\end{align}
Now, by a result of Haagerup, for every $\alpha \in (0,1)$, the mapping
\begin{align*}
  \varphi_\alpha: \mathbb{F}_n\to \mathbb{R} \ \ , \ \ \varphi_p(s)=\alpha^{|s|} \ \ (s\in \mathbb{F}_n),
\end{align*}
is a positive-definite function on $\mathbb{F}_n$. Here $|\cdot|$ is the standard word-length function on $\mathbb{F}_n$.
Hence, if $\pi_\alpha$ is the GNS-representation of $\varphi_\alpha$, then we have
\begin{align*}
  w\left(\sum_{i=1}^{n} \pi_\alpha(s_i) \right) & = w\left( \pi_\alpha\left(\sum_{i=1}^{n} \delta_{s_i}\right) \right) \\
  &\geq \varphi_\alpha\left(\sum_{i=1}^{n} \delta_{s_i}\right) \\
  &=\sum_{i=1}^{n} \alpha^{|s_i|} \\
  &=n\alpha,
\end{align*}
as $|s_i|=1$ for every $i=1,\dots,n$. By looking at the growth of $ \mathbb{F}_n$, it is routine to check that $\varphi_\alpha\in \ell^p( \mathbb{F}_n)$ if $\alpha<(2n-1)^{-1/p}$. This implies that if the relation \eqref{Eq:1} always holds, then combining the above with the relation \eqref{Eq:2} and the fact that $\pi_\alpha$ is an $\ell^p$-representation for every  $\alpha<(2n-1)^{-1/p}$, we must have
\begin{align*}
   n(2n-1)^{-1/p}& =\lim_{\alpha\to (2n-1)^{-1/p}} n\alpha \\
   &\leq \lim_{\alpha\to (2n-1)^{-1/p}}  w\left(\sum_{i=1}^{n} \pi_\alpha(s_i) \right) \\
   &\leq \lim_{\alpha\to (2n-1)^{-1/p}}  w\left(\sum_{i=1}^n \pi_\alpha(s_i) \otimes \bar{\pi}_\alpha(s_i)\right)\\
  & \leq 2n^{1-4/p}(2n-1)^{2/p}.
\end{align*}
In short, we will have
\begin{align*}
  n^{4/p} \leq 2(2n-1)^{3/p}\Rightarrow n \leq 2^{p+3}.
\end{align*}
This is, of course, a contradiction based on our hypothesis for $n$ and $p$.
\end{proof}

\

\section{Acknowledgements}
The authors acknowledge the support of the Institut Mittag-Leffler during the preparation of this manuscript. This material is based upon work supported by the Swedish Research Council under grant
no. 2021-06594 while the authors were in residence at Institut Mittag-Leffler in Djursholm, Sweden during the year of 2026. M.R. was partially
supported by the Wallenberg Centre for Quantum Technology (WACQT) funded by the Knut and Alice Wallenberg Foundation (KAW). 
E.S. would like to thank the Wenner-Grenn Foundation (the project GFOh2024-0025) which supported his sabbatical visit to Chalmers University of Technology, 2025-2026. E.S. was partially supported by NSERC Discovery Grant RGPIN-2025-04833.

\end{document}